\documentclass{article}
\usepackage{amsmath,amssymb,amsthm}
\usepackage{graphics}

\usepackage{graphicx}

\date{}

\newcommand{\R}{\ensuremath{\mathbb{R}}}

\newtheorem{thm}{Theorem}[section]
\newtheorem{defn}{Definition}[section]
\newtheorem{lem}{Lemma}[section]

\newtheorem{clm}{Claim}[section]

\begin{document}

\title{The Extremal Spheres Theorem}

\author{Arseniy Akopyan\thanks{Research supported in part by RFBR grants 08-01-00565-a and 10-01-00096-a}
\and Alexey Glazyrin\thanks{Research supported in part by RFBR grant 08-01-00565-a}
\and Oleg R. Musin \thanks{Research supported in part by NSF
grant DMS-0807640 and NSA grant MSPF-08G-201.}
\and Alexey Tarasov\thanks{Research supported in part by RFBR grant 08-01-00565-a}}


\maketitle

\begin{abstract}
Consider a polygon $P$ and all neighboring circles (circles going
through three consecutive vertices of~$P$). We say that a
neighboring circle is extremal if it is empty (no vertices of $P$
inside) or full (no vertices of $P$ outside). It is well known
that for any convex polygon there exist at least two empty and at
least two full circles, i.e. at least four extremal circles. In
1990 Schatteman considered a generalization of this theorem for
convex polytopes in $d$-dimensional Euclidean space. Namely, he
claimed that there exist at least $2d$ extremal neighboring
spheres for generic polytopes. His proof is based on the Bruggesser-Mani shelling method.

In this paper, we show that there are certain gaps in Schatteman's
proof.  We also show that using the Bruggesser-Mani-Schatteman method it is possible to prove that there are at least $d+1$ extremal neighboring  spheres. However, the
existence problem of $2d$ extremal neighboring spheres is still
open.
\end{abstract}

\section{Introduction}

In 1990 Schatteman \cite{Sch} published a four-vertex type theorem for polyhedrons. 
We must emphasize that this was a unique generalization of the four-vertex theorem for higher dimensions.

Consider a $d$-dimensional simplicial polytope $P$ in
$d$-dimensional Euclidean space. We call this polytope generic if
it has no $d+2$ cospherical vertices and is not a $d$-dimensional
simplex. From now on, we consider only generic simplicial
polytopes. Each $(d-2)$-dimensional face uniquely defines a
neighboring sphere going through the vertices of two facets
sharing this $(d-2)$-dimensional face. A neighboring sphere is
called empty if it does not contain other vertices of $P$ and it
is called full if all other vertices of $P$ are inside of it. We
call an empty or full neighboring sphere extremal. Schatteman
(\cite{Sch}, Theorem 2, p.232) proves the following claim.

\begin{clm} [Schatteman, 1990]
\label{thm:Schatteman} For any convex  $d$-dimensional polytope
$P$ there are at least $d$ different $(d-2)$-dimensional faces defining
empty neighboring spheres and at least $d$ different $(d-2)$-dimensional
faces defining full neighboring spheres.
\end{clm}

This result for $d=2$ is well known (see Section 2). Actually, it immediately follows from Theorem 3.1 in Section 3 of our paper. Unfortunately, already for d=3 the proof
by Schatteman contains a crucial gap (see Section 6). At the moment we do not know if the claim is correct.

In Section 7, using Bruggesser-Mani-Schatteman's method, we
show that there are at least two different $(d-2)$-dimensional faces defining
empty neighboring spheres and at least two different $(d-2)$-dimensional faces defining full neighboring spheres. Moreover, we prove that there are at least $d+1$ extremal (full and empty) neighboring  spheres.

In his paper, Schatteman did not necessarily consider simplicial
and generic polytopes. The lack of these conditions can only add
some minor technical difficulties (for instance, the neighboring
sphere is not well defined by a non-simplicial $(d-2)$-dimensional
face and Delaunay triangulation is not determined without the
generic assumption), which do not affect the main arguments. So
in this paper we consider only simplicial and generic polytopes.

\section{Four-vertex type theorems for polygons}

We define an {\it oval\/} as a convex smooth closed plane curve.
The classical four-vertex theorem by Mukhopadhayaya \cite{Mukh}
published in 1909 says the following: {\it The curvature function
on an oval has at least four local extrema (vertices).} It is well
known that any continuous function on a compact set has at least
two (local) extrema: a maximum and a minimum. It turns out that
the curvature function has at least four local extrema. The paper
was noticed, and generalizations of the result appeared almost
immediately. In 1912, A.~Kneser \cite{AKn} showed that convexity
is not a necessary condition and proved the four-vertex theorem
for a simple closed plane curve.

The famous book \cite{Bla} by W.~Blaschke (first published in
1916), together with other generalizations, contains a
``relative'' version of the four-vertex theorem. Here we preserve
the formulation and notation from \cite{Bla}. {\it Let $C_1$ and
$C_2$ be two (positively oriented) convex closed curves, and let
$do_1$ and $do_2$ be arc elements at points with parallel (and
codirected) support lines. Then the ratio $do_1/do_2$ has at least
four extrema.} In the case where $C_2$ is a circle, this theorem
becomes the theorem on four vertices of an oval.

In 1932, Bose~\cite{Bose} published a remarkable version of the
four-vertex theorem in a global sense. While in the classical
four-vertex theorem the extrema are defined ``locally,'' here they
are defined ``globally.'' Let $G$ be an oval such that no four
points lie on a circle. We denote by $s_-$ and $s_+$
($\text{resp}$., $t_-$ and  $t_+$) the number of its circles of
curvature ($\text{resp}$., the circles which are tangent to $G$ at
exactly three points) lying inside ($-$) and outside ($+$) the
oval $G$, respectively (the {\it curvature circle\/} of $G$ at a
point $p$ is tangent to $G$ at $p$ and has radius $1/k_G(p)$,
where $k_G(p)$ is the curvature of $G$ at $p$). In this notation,
we have the relation $s_- - t_- = s_+ - t_+ = 2.$ If we define
{\it vertices\/} as the points of tangency of the oval $G$ with
its circles of curvature lying entirely inside or outside $G$,
then these formulas imply that the oval $G$ has at least four
vertices. It is worth mentioning that this fact was proved by
H.~Kneser \cite{HKn} ten years before Bose.

Since then, publications related to the four-vertex theorem did
not halt and their number considerably increased throughout recent
years (see~\cite{Ar1,Ar2}, etc.) to a large extent due to papers
and talks by V.~I.~Arnold. In the above papers, various versions
of the four-vertex theorem for plane curves and convex curves in
${\mathbb R}^d$, and their special points (vertices) are
considered: critical points of the curvature function, flattening
points, inflection points, zeros of higher derivatives, etc. The
paper by Umehara \cite{Ume} and the book \cite{Pak} by Pak  contain long lists of
papers devoted to these topics. Several interesting results were
obtained by Tabachnikov (\cite{Tab1}, \cite{Tab2}) also in this
direction.

It is interesting to note that the first discrete analog of the
four-vertex theorem arose almost 100 years before its smooth
version. In 1813, the splendid paper by Cauchy on rigidity of
convex polyhedra used the following lemma: {\it Let $M_1$ and
$M_2$ be convex $n$-gons with sides and angles $a_i$, $\alpha_i$
and $b_i$, $\beta_i$, respectively. Assume that $a_i = b_i$ for
all $i=1,\dots,n$. Then either $\alpha_i = \beta_i$, or the
quantities $\alpha_i - \beta_i$ change sign for $i=1,\dots,n$  at
least four times.}

In Aleksandrov's book \cite{Al}, the proof of uniqueness of a
convex polyhedron with given normals and areas of faces involves a
lemma where the angles in the Cauchy lemma are replaced by the
sides. We present a version of it, which is somewhat less general
than the original one. {\it Let $M_1$ and $M_2$ be two convex
polygons on the plane that have respectively parallel sides.
Assume that no parallel translation puts one of them inside the
other. Then when we pass along $M_1$ (as well as along $M_2$), the
difference of the lengths of the corresponding parallel sides
changes the sign at least four times.}

We easily see the resemblance between the above relative
four-vertex theorem for ovals (apparently belonging to Blaschke)
and the Cauchy and Aleksandrov lemmas. Furthermore, approximating
ovals by polygons, we can easily prove the Blaschke theorem with the
help of any of these lemmas.

The Cauchy and Aleksandrov lemmas easily imply four-vertex
theorems for a polygon:

\medskip

\noindent({\bf Corollary of the Cauchy lemma}) {\it Let
$M$ be an equilateral convex polygon. Then at least two
angles of $M$ do not exceed the neighboring angles, and at least
two angles of $M$ are not less than the neighboring
angles.}

\medskip

\noindent({\bf Corollary of the Aleksandrov lemma}) {\it Let all angles of a
polygon $M$ be pairwise equal. Then at least two sides of
$M$ do not exceed their neighboring sides, and at least two
sides of $M$ are not less than their neighboring sides.}

\medskip

In applications, the {\it curvature radius\/} at a vertex of a
polygon is usually calculated as follows. Consider a polygon $M$
with vertices $A_1, \dots, A_n$. Each vertex $A_i$ has two
neighbors: $A_{i-1}$ and $A_{i+1}$. We define the curvature radius
of $M$ at $A_i$ as follows: $R_i(M)$ is equal to the circumradius
of $\triangle A_{i-1}A_{i}A_{i+1}.$

\begin{thm}[\cite{Mus1}]
Assume that $M$ is a convex polygon and that for each vertex $A_i$
of $M$, the circumcenter of $\triangle A_{i-1}A_{i}A_{i+1}$ lies
inside the angle $\angle{A_{i-1}A_{i}A_{i+1}}$. Then the theorem
on four local extrema holds true for the (cyclic) sequence of the
numbers $R_1(M)$, $R_2(M)$, \dots, $R_n(M)$, i.e., at least two of
the numbers do not exceed the neighboring ones, and at least two
of the numbers are not less than the neighboring ones.
\end{thm}

The generalization of this theorem for the case of non-convex
polygons was given by V.~D.~Sedykh \cite{Sed1}, \cite{Sed2}.
Furthermore, this theorem generalizes the four-vertex theorems
following from the Cauchy and Aleksandrov lemmas.

A circle $C$ passing through certain vertices of a polygon $M$ is
said to be {\it empty\/} (respectively, {\it full\/}) if all the
remaining vertices of $M$ lie outside (respectively, inside) $C$.
The circle $C$ is {\it extremal\/} if $C$ is empty or full.

\begin{thm}[Folklore]
\label{thm:circles}
Let $M=A_1\dots A_n$ be a convex $n$-gon,
$n>3$, no four vertices of which lie on one circle. Then at least
two of the $n$ circles $C_i(M)$ (the circumcircle
 of $\triangle A_{i-1}A_{i}A_{i+1})$, $i=1,\dots,n$, are empty and at least two
of them are full, i.e., there are at least four extremal circles.
\end{thm}

(S. E. Rukshin told one of the authors that this result for many
years has been included in the list of problems for training for
mathematical competitions and is well known to St. Petersburg
school students attending mathematical circles.)

Theorem \ref{thm:circles} was also generalized for the case of
non-convex polygons by Sedykh \cite{Sed1}, \cite{Sed2}.

It is easy to see a direct generalization of the Bose theorem for
polygons.

{\it We denote by $s_-$ and $s_+$ the numbers of empty and full
circles among the circles $C_i(M)$, and we denote by $t_-$ and
$t_+$ the numbers of empty and full circles passing through three
pairwise non-neighboring vertices of $M$, respectively. Then, as
before, we have $ s_- - t_- = s_+ - t_+ = 2.$}

This fact was suggested by Musin as a problem for the All-Russia
mathematics competition of high-school students in 1998. It will
be proved in the next section by means of elementary planar
geometry methods.

One more generalization of the Bose theorem is given in
\cite{Weg}, where one considers the case of an {\it equilateral\/}
polygon, which is not necessarily convex.

V.~D.~Sedykh \cite{Sed1} proved a theorem on four support planes
for weakly convex polygonal lines in ${\mathbb R}^3$: {\it If any
two neighboring vertices of a polygon $M$ lie on an empty circle,
then at least four of the circles $C_i(M)$ are extremal.} It is
clear that convex and equilateral polygons satisfy this condition.
Furthermore, Sedykh constructed examples of polygons showing that
his theorem is wrong without this assumption (see \cite{Sed2}).

This short survey cannot be considered  as complete. For more details see \cite{WM,Mus3,Pak,Ume}.

\section{The Bose theorem for polygons}

Here we provide an elementary proof for the two-dimensional case of
the extremal spheres theorem. The proof of all statements from this section is the same for
empty and full circles, so without loss of generality, we consider only empty circles.

To be consistent with
\cite{Pak} we call a circle through three vertices of a polygon
{\it neighboring} if these three vertices are consecutive, {\it disjoint} if
there are no adjacent vertices among these three and {\it intermediate}
in all other cases.

\begin{thm}
Let $Q \in R^2$ be generic convex polygon with $n$ vertices and $n
\geq 4$. Denote by $s_+, t_+$ and $u_+$ the number of full circles
that are neighboring, disjoint and intermediate, respectively.
Similarly, denote by $s_-, t_-$ and $u_-$ the number of empty
circles that are neighboring, disjoint and intermediate,
respectively. Then
$$s_+ - t_+ = s_- - t_- = 2,$$ $$s_+ + t_+ + u_+ = s_- + t_- + u_-
= n-2$$
\end{thm}

\begin{proof}
Let us prove that triangles with empty circumcircles form a
triangulation of $Q$. The proof is based on two lemmas.

\begin{lem}
Triangles with empty circumcircles do not intersect
\end{lem}

\includegraphics[scale=1]{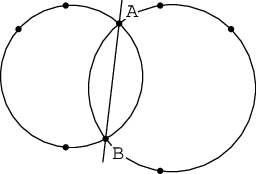}

\begin{proof}
If the respective circumcircles $\omega_1, \omega_2$ corresponding to each triangle do not intersect, then
the statement of this lemma is obvious. So we assume they
intersect in points $A$ and $B$. All vertices of the first triangle lie
on the arc of $\omega_1$ outside of $\omega_2$ and all vertices of
the second triangle lie on the arc of $\omega_2$ outside of
$\omega_1$. But, these two arcs lie in different half-planes with
respect to line $AB$. Thus the triangles do not intersect.
\end{proof}

\begin{lem}
If triangle $ABC$ has an empty circumcircle and $BC$ is a
diagonal of $Q$, then there exists a triangle $BCD$ with an empty
circumcircle.
\end{lem}

\includegraphics[scale=1]{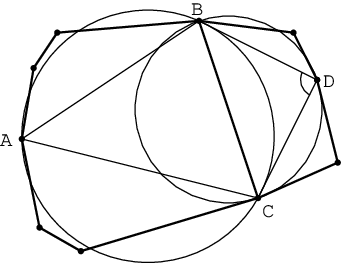}

\begin{proof}
Consider all angles $\angle BFC$ such that $F$ is a vertex of $Q$
and is in a different half-plane than $A$ with respect to $BC$.
Suppose $\angle BDC$ is the maximal angle in this set (it is
unique due to the generic assumption). Then obviously triangle
$BCD$ is the triangle required.
\end{proof}

From these two lemmas it follows that triangles with empty circles
form a triangulation.

The total number of triangles in a triangulation of a polygon with $n$ vertices is always $n-2$,
so the second part of the theorem is proved.

Each triangle corresponding to a neighboring circle has two edges
of $Q$ as its own edges. Each triangle corresponding to an
intermediate circle has one edge of $Q$ as its own edge. Each
triangle corresponding to a disjoint circle has no edges of $Q$ as
its own edges. Hence, $2 s_- + u_- = n$. We subtract $s_- + t_- +
u_- = n-2$ from this equality and obtain that $s_- - t_- = 2$.
\end{proof}

From this theorem we can obtain the theorem on extremal circles.
We have that $s_- - t_- = 2$ and therefore $s_- \geq 2$.
Analogously, $s_+ \geq 2$ proving that each generic convex polygon
possesses at least two full and at least two empty neighboring
circles.

The proof of this theorem has a very reasonable explanation in
terms of Delaunay triangulations which we will consider in the next section.

\section{Constructing Delaunay triangulations by lifting to a spherical paraboloid}

Suppose $S$ is a set of points in $\R^d$, the affine dimension of $S$
is $d$ and there are no $d+2$ cospherical points in $S$.
We define a Delaunay (upper Delaunay) simplex as a simplex, whose circumsphere
contains no (all) points of $S$.
A Delaunay (upper Delaunay) triangulation of $S$ is a triangulation of $S$ consisting of all Delaunay (upper Delaunay) simplices.
In this section, we show how these triangulations can be constructed (and simultaneously prove
that they exist) for the set of vertices of a $d$-dimensional convex
generic simplicial polytope.

Consider the set $S$ of all vertices of the $d$-dimensional polytope
$P$. Let $DT(S)$ denote a Delaunay triangulation of the set $S$
and $UDT(S)$ be an upper Delaunay triangulation of $S$.

Here we use a method of constructing $DT(S)$ and $UDT(S)$ by
lifting all the points of $S$ to a spherical paraboloid. The idea
of this construction belongs to Voronoi \cite{Vor}. We define a
function $f:\R^d\longrightarrow \R^{d+1}$ such that for
$X(x_1,x_2,\ldots,x_d) \in \R^d$, its image
$f(X)=(x_1,x_2,\ldots,x_d,x_1^2+x_2^2+\ldots+x_d^2)$. Then, for
all points $X_1,X_2,\ldots,X_n \in S$, we consider the set $S'$ of
$f(X_1),f(X_2),\ldots,f(X_n)$. Because of the generic assumption
on $P$ there are no $d+2$ points of $S'$ on the same
$d$-dimensional hyperplane. Let $CH(S')$ be the convex hull of the
set $S'$. Obviously, $CH(S')$ is a simplicial polytope and each
$f(X_i)$ is a vertex of this polytope. For each facet $F_j$ of
$CH(S')$, consider the exterior normal of this facet $n_j$. We
then divide all facets into two groups subject to the following
rule: if the last coordinate of $n_j$ is negative, we place $F_j$
in the first group and if it is positive, we place $F_j$ in the
second group (since $P$ is simplicial, it can never be equal to
0).

\begin{lem}
Projections of all facets of the first group to $\R^d$ give
$DT(S)$. Projections of all facets of the second group to
$\R^d$ give $UDT(S)$.
\end{lem}

\begin{proof}
Consider a facet $F$ from the first group. Suppose the equation of
the $d$-dimensional hyperplane containing $F$ is $x_{d+1}=a_1 x_1
+ \ldots + a_d x_d$. Consider the intersection of this hyperplane
with the paraboloid $x_{d+1}=x_1^2+ \ldots + x_d^2$. Thus, $a_1
x_1 + \ldots + a_d x_d = x_1^2+ \ldots + x_d^2$ and $(x_1 - \frac
{a_1} 2)^2 + \ldots + (x_d - \frac {a_2} 2)^2 = \frac {a_1^2} 4 +
\ldots + \frac {a_d^2} 4$. Hence the projection of the
intersection of this hyperplane with the paraboloid is a sphere.
Now we notice that $F$ is from the first group so for all vertices
of $S'$ on the paraboloid we have that $x_{d+1}
\geq a_1 x_1 + \ldots + a_d x_d$. So, for their projections $(x_1 -
\frac {a_1} 2)^2 + \ldots + (x_d - \frac {a_2} 2)^2 \geq \frac
{a_1^2} 4 + \ldots + \frac {a_d^2} 4$. Thus all the vertices are
outside of the circumsphere of the projection of $F$, so the
projection of $F$ is a Delaunay simplex.

The proof is exactly the same for $UDT(S)$.
\end{proof}

Notice that because of convexity of $P$ $DT(S)$ and $UDT(S)$ are
also triangulations of $P$. So further we will call them Delaunay
triangulation of $P$ and upper Delaunay triangulation of $P$.

\section{Different types of ``ears'' and connections between them}

\begin{defn}
Consider a convex generic simplicial polytope $P$ and its Delaunay
triangulation $DT(P)$ (upper Delaunay triangulation $UDT(P)$). We
say that a simplex $S \in DT(P)$ ($S \in UDT(P)$) is a D-ear
(UD-ear) if at least two facets of $S$ are on the boundary of $P$.
\end{defn}

There is a direct connection between ears and extremal spheres:
the sphere circumscribed around any ear is extremal and the
simplex on the vertices of $P$ inscribed in the extremal sphere is
an ear.

In his paper, Schatteman uses the concept of shellability. Let us give
a formal definition of a polytopal complex and its shelling (see
also \cite{Zie}).

{\it A polytopal complex} $C$ in $\R^d$ is a collection of
polytopes in $\R^d$ such that 1) the empty set is in $C$, 2) for
any polytope $T \in C$ every face of $T$ is also in $C$, 3) the
intersection of any two polytopes in $C$ is a face of both.

The dimension of $C$ is the largest dimension of a polytope in
$C$. Inclusion-maximal faces of a complex are called facets. If
all facets are of the same dimension then we call $C$ pure.

Every $0$-dimensional complex is called shellable, its shelling is
any ordering of facets. Let $C$ be a pure $d$-dimensional
polytopal complex. $C$ is called {\it shellable} if there exists
an ordering ({\it shelling}) of its facets $(F_1,F_2,\ldots,F_m)$
such that $\forall s: 2 \leq s \leq m; (\bigcup
\limits_{t=1}^{s-1} F_t)\bigcap F_s$ is a beginning of a shelling
of $\partial F_s$. If complex $C$ is a $d$-cell then each partial
union $\bigcup \limits_{t=1}^{s} F_t$, $1\leq s \leq m$, of its
shelling is homeomorphic to a $d$-dimensional disk (if it is a
$d$-sphere, then $\bigcup \limits_{t=1}^{s} F_t$ is homeomorphic
to a disk for all $s<m$ and to a sphere for $s=m$).

%

We say that facet $F$ of a polytope $P$ is {\it visible} from
point $A$ if $A$ and $P$ are in different open half-spaces with
respect to the hyperplane through $F$. In the classical paper by
Bruggesser and Mani \cite{BM} it was proved that the complex of
facets of some polytope visible from some point is shellable, and
the method of constructing a shelling is given. Connect this point
with some internal point of the polytope and the order by which
hyperplanes determined by the facets of the polytope intersect
this segment (starting from interior of the polytope) is the order
of corresponding simplices in a shelling.

This method can be applied to the simplicial complex of $DT(P)$.
We use the method of lifting to a spherical paraboloid from the
previous section. Suppose we obtain a $(d+1)$-dimensional polytope
$P'$ from this procedure. Then we know that the lower facets
(facets for which the last coordinates of normals are negative)
correspond to simplices in $DT(P)$. In fact, the simplicial
complex of $DT(P)$ and the complex of lower facets are isomorphic.
Obviously, there exists a sufficiently low point in $\R^{d+1}$
such that the set of facets visible from it is exactly the set of
lower facets. Thus the following lemma is true:

\begin{lem}[\cite{Sch}, Lemma 1, p.237]
The Delaunay (upper Delaunay) simplicial complex is shellable.
\end{lem}

In his paper, Schatteman proves the following lemma:

\begin{lem}[\cite{Sch}, Lemma 2, p.237]
The last simplex in the Delaunay (upper Delaunay) shelling is a
D-ear (UD-ear).
\end{lem}

\begin{proof}
Consider the last simplex $T_m$ from the Delaunay shelling.
$(\bigcup \limits_{t=1}^{m-1} T_t), (\bigcup \limits_{t=1}^{m}
T_t) = P$ and $T_m$ are homeomorphic to a $d$-dimensional disk.
Hence $\partial T_m \bigcap \partial P$ must be homeomorphic to a
$(d-1)$-dimensional disk. But this common bound contains all
vertices of $T_m$. Thus it contains at least two facets of $T_m$,
so $T_m$ is an ear.

The proof for the case of upper Delaunay is the same.
\end{proof}

\begin{defn}
The last simplex of some shelling of $DT(P)$ ($UDT(P)$) obtained
by the method of Bruggesser and Mani for $P'$ is called a BMD-ear
or Delaunay BM-ear (BMUD-ear or upper Delaunay BM-ear).
\end{defn}

\section{The status of Schatteman's theorem}

Here we briefly observe Schatteman's proof of his theorem and show
that there are gaps that cannot be filled by his ideas.

In order to prove the theorem, Schatteman considered two cases: 1)
each Delaunay simplex is a D-ear, 2) there is a Delaunay simplex
which is not a D-ear.

The proof for the second case was based on the following claim:

\begin{clm}
If there exists Delaunay simplex which is not a BMD-ear then there
are at least $d$ BMD-ears.
\end{clm}

(Although this claim was not formulated in his paper, Schatteman
tries to prove it and uses it in the proof of his main theorem.)

We show that the proof of this claim is wrong. Moreover, we
give a counterexample in dimension three. We think that this
counterexample can be generalized for higher dimensions and
consequently the claim is completely false.

His proof was the following. He lifted vertices of $P$ to a
paraboloid and constructed the convex hull $P'$ of the set of
lifted vertices. Then, for all lower facets of $P'$, he
constructed a polyhedron $E$ defined by the intersection of open
half-spaces determined by the facets and not containing $P'$. He
considered the vertex of $E$ and tried to find a shelling by the
Bruggesser and Mani method using a line connecting this vertex and
the interior of $P'$. The problem here is that the Bruggesser-Mani
method allows us to find a shelling only for the set of facets
visible from a given point. From the chosen point in Schatteman's
proof, there can be some upper facets that are visible. Hence his
method cannot give us a shelling of the complex of lower facets
which correspond to Delaunay cells.

Actually there are even configurations that contradict this
claim. Here is an example of the three-dimensional polytope
that has only two BMD-ears.

\includegraphics[scale=1]{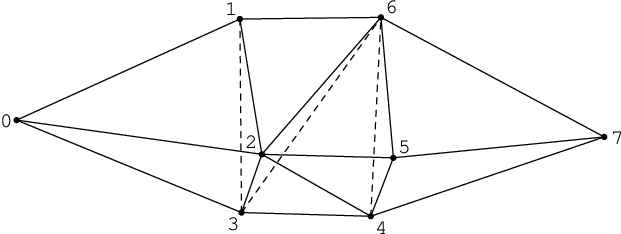}

As it can be seen from the Delaunay triangulation of this
polytope, it has five D-ears. However only two of them are
BMD-ears -- $(0,1,2,3)$ and $(4,5,6,7)$.

\includegraphics[scale=1]{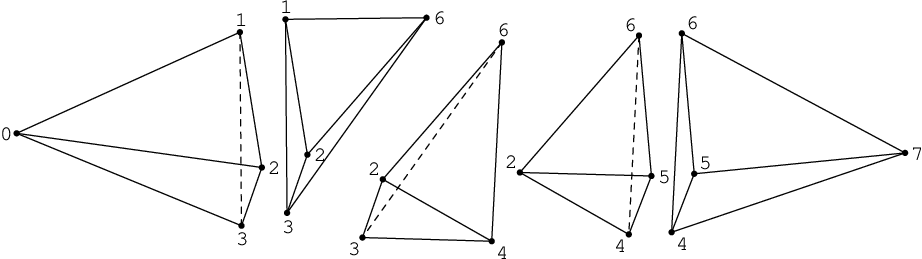}

Thus, the existence problem of $2d$ extremal neighboring spheres
is still open.

\section{Theorems on the number of ears}

In this section we prove results pertaining to the number of ears that can be
obtained using Schatteman's method.

\begin{thm}\label{thm:2 BM-ears}
For each generic convex simplicial $d$-dimensional polytope there
exist at least two Delaunay and at least two upper Delaunay
BM-ears.
\end{thm}

\begin{proof}
We consider a $(d+1)$-dimensional polytope $P'$ that is the result
of lifting $P$ to a spherical paraboloid. As in the previous
sections, the lower facets of $P'$ give $DT(P)$ and the upper
facets give $UDT(P)$. We define the unbounded polyhedron $P_+$ as
the intersection of all upper half-spaces with all facets of $P'$.
Analogously, we define $P_-$ as the intersection of all lower
half-spaces with all facets of $P'$. Obviously the facet of $P_-$
defined by the lower facet of $P'$ is the last in some
Bruggesser-Mani shelling and corresponds to a BMD-ear of $P$.
Similarly, the facet of $P_+$ defined by the upper facet of $P'$
is the last in some Bruggesser-Mani shelling and it corresponds to
a BMUD-ear of $P$. Let us now show that there exist at least two
facets of $P_-$ defined by lower facets of $P'$ (for $P_+$ the
proof is the same). We take any point inside $P'$ and the vertical
ray from this point to $x_{d+1}=-\infty$. The last hyperplane
intersecting this ray is a facet of $P_-$ and it cannot be a
hyperplane defined by some upper facet of $P'$. Thus we already
have one BMD-ear. Now let $F$ be the facet of $P'$ defining the
facet of $P_-$ observed earlier. Let us take a point inside $P'$
such that the vertical ray from this point $x_{d+1}=-\infty$
intersects only $F$ out of all facets of $P'$. Then the last
hyperplane intersecting this ray is not a hyperplane defined by
$F$ and it gives us one more BMD-ear.

\end{proof}

\begin{thm}
For each convex generic simplicial $d$-dimensional polytope $P$
there exist at least $d+1$ BM-ears.
\end{thm}

\begin{proof}
As in the proof of the previous theorem, we are
interested in facets of $P_+$ defined by upper facets of $P'$ and
facets of $P_-$ defined by lower facets of $P'$.

For each facet of $P'$ consider a hyperplane parallel to this
facet and passing through the origin. Consider the unbounded
polyhedron $T_+$ given by the intersection of all upper
half-spaces defined by all such hyperplanes.

\begin{lem}
For each facet of $T_+$ there is an unbounded facet of $P_+$
parallel to it.
\end{lem}

\begin{proof}
Suppose $f$ is a facet of $T_+$. There is a ray $r$ in $f$ which
is strictly in the upper half-space for all other hyperplanes.
Hyperplane $f$ corresponds to one or two facets of $P'$ (there
cannot be more than two parallel facets for one polytope).
Consider the highest of these facets and the ray $r'=r$ (equality is
given with respect to parallel translations) lying on it. It is obvious
that, for a hyperplane defined by some facet $h$ of $P'$, the unbounded
part of $r'$ (subray $r_h$ of $r'$) lies in the upper half-space,
i.e. only a segment of this ray can be in a lower half-space. So,
$\widehat{r}=\bigcap_h r_h$ is a ray which is in our hyperplane
and is in an upper half-space for all hyperplanes defined by facets
of $P'$. This means that $f$ defines an unbounded facet of $P_+$.
\end{proof}

We can define $T_-$ as the intersection of lower half-spaces. The
same lemma connecting facets of $T_-$ and $P_-$ is true by the
same arguments. $T_+$ and $T_-$ are symmetric with respect to the
origin and have at least $d+1$ facets. By
this lemma and its analogue for $T_-$, for each of these facets there is an unbounded facet
of $P_+$ and $P_-$ parallel to it. Thus, if a facet of $P'$
defining a facet of $T_+$ were to be a lower face, then we have a BM-ear from
$P_-$, if it were to be an upper face, then we have a BM-ear from $P_+$, and if
there were to be two parallel facets defining this facet of $T_+$, then we
have BM-ears from both $P_-$ and $P_+$. So for each facet of $T_+$
we have at least one corresponding BM-ear. It means that there exist at
least $d+1$ BM-ears.
\end{proof}

\section{Remarks}

Here we want to make a remark on non-convex polytopes. The
algorithm of Section 4 allows us to construct Delaunay
triangulations of sets of points. These triangulations are always
subdivisions of the convex hull of the points into simplices.
Hence using the ideas of Section 5 (and further Section 7) we
can obtain only extremal spheres of the convex hulls but not of
the actual polytopes in the case where these polytope are not convex. For
non-convex polytope there is no direct connection between its
subsimplices and facets of $P'$ constructed by means of lifting to
a paraboloid, since facets of $P'$ correspond to subsimplices of
the convex hull. Maybe for some special cases (for instance there
exists a triangulation of the polytope which is a subset of the
Delaunay triangulation of the set of points) partial results can
be achieved but in the general case the proofs in Sections 5 and 7
do not work.

\medskip

\noindent{\bf\large Acknowledgements.} The authors are grateful to Wiktor Mogilski for his extremely
helpful comments and corrections.

\medskip

A. Akopyan,  A. Glazyrin, O. R. Musin, A. S. Tarasov,  Department of Mathematics, University of Texas at Brownsville, 80 Fort Brown, Brownsville, TX, 78520.

 {\it E-mail addresses:} akopjan@gmail.com alexey.glazyrin@gmail.com oleg.musin@utb.edu tarasov@isa.ru

\end{document}